\documentclass[11pt]{article}

\usepackage[margin=1.1in]{geometry}
\usepackage{amsmath,amssymb,amsthm}
\usepackage{mathtools}
\usepackage[hidelinks]{hyperref}

\newtheorem{theorem}{Theorem}
\newtheorem{lemma}[theorem]{Lemma}
\newtheorem{corollary}[theorem]{Corollary}

\newtheorem{conjecture}[theorem]{Conjecture}

\newcommand{\chieq}{\chi_=}
\newcommand{\amin}{\alpha_{\min}}
\newcommand{\ceil}[1]{\left\lceil #1 \right\rceil}
\newcommand{\floor}[1]{\left\lfloor #1 \right\rfloor}

\title{A disproof of a gap-one conjecture for the equitable chromatic number of block graphs}
\author{Juho Lauri\thanks{Corresponding author. Helsinki, Finland.
Email: juho.lauri@gmail.com}}
\date{}

\begin{document}

\maketitle

\begin{abstract}
For a graph \(G\), let
\[
  L(G)=
  \max\left\{
    \omega(G),
    \ceil{\frac{|V(G)|+1}{\amin(G)+1}}
  \right\},
\]
where \(\omega(G)\) denotes the clique number, and
\[
  \amin(G)=\min_{v\in V(G)}
  \max\{|I|\mid I\text{ is independent and }v\in I\}.
\]
Dybizba\'nski, Furma\'nczyk, and Mkrtchyan (Discrete Appl.\ Math.\ 354
(2024), 15--28) conjectured that every block graph \(G\) satisfies
\(L(G)\le \chieq(G)\le L(G)+1\), where \(\chieq(G)\) is the equitable
chromatic number of \(G\).
We disprove this conjecture in a strong form.  For every pair of integers
\(d\ge2\) and \(k\ge4d-1\), we construct a connected block graph \(G_{d,k}\)
such that
\(L(G_{d,k})=k\) and \(\chieq(G_{d,k})=k+d\).  Thus the difference
\(\chieq(G)-L(G)\) is unbounded on connected block graphs.
\end{abstract}


\section{Introduction}

All graphs in this note are finite and simple, and all numbers of colors are
positive integers.

An \emph{equitable \(q\)-coloring} of a graph \(G\) is a proper vertex coloring
with \(q\) colors in which all color classes have sizes
\[
  \floor{\frac{|V(G)|}{q}}
  \quad\text{or}\quad
  \ceil{\frac{|V(G)|}{q}}.
\]
The \emph{equitable chromatic number} \(\chieq(G)\) is the least \(q\) for
which such a coloring exists.  Equitable coloring is a classical variant of
graph coloring.  A foundational theorem of Hajnal and
Szemer\'edi~\cite{HajnalSzemeredi} states that every graph of maximum degree
\(\Delta\) has an equitable \((\Delta+1)\)-coloring.

Let \(\alpha(G)\) denote the independence number of \(G\), that is, the
maximum size of an independent set in \(G\).  For a vertex \(v\in V(G)\), let
\[
  \alpha(G,v)=
  \max\{|I|\mid I\text{ is independent and }v\in I\},
\]
and let
\[
  \amin(G)=\min_{v\in V(G)}\alpha(G,v).
\]
Every graph satisfies
\[
  \chieq(G)\ge
  L(G)=
  \max\left\{
    \omega(G),
    \ceil{\frac{|V(G)|+1}{\amin(G)+1}}
  \right\}.
\]
Indeed, the clique-number bound \(\chieq(G)\ge\omega(G)\) is immediate.  For the second term, choose \(v\) with
\(\alpha(G,v)=\amin(G)\).  In an equitable \(q\)-coloring, the class containing
\(v\) has size at most \(\amin(G)\), and every other class has size at most
\(\amin(G)+1\).  Hence \(|V(G)|\le q(\amin(G)+1)-1\).

Block graphs form a natural test class for equitable coloring.  A
\emph{block graph} is a graph whose blocks are cliques.  They are chordal and
tree-like in structure, but equitable coloring is algorithmically hard on this
class.  In fact, equitable \(q\)-colorability is \(\mathrm{W}[1]\)-hard on
block graphs when parameterized by \(q\)
\cite{GomesLimaSantosParameterized}.  Related structural parameterizations of
equitable coloring were studied by Gomes, Guedes, and
dos Santos~\cite{GomesGuedesSantosLATIN,GomesGuedesSantosAlgorithmica}.
Dybizba\'nski, Furma\'nczyk, and
Mkrtchyan~\cite{DybizbanskiFurmanczykMkrtchyan} observed that the lower bound
\(L(G)\) need not be exact, but that
their examples missed it by only one color.  Motivated by this
Vizing--Goldberg-type phenomenon of a natural lower bound being off by at most
one, they made the following conjecture.

\begin{conjecture}[Dybizba\'nski--Furma\'nczyk--Mkrtchyan~\cite{DybizbanskiFurmanczykMkrtchyan}]
\label{conj-gap-one}
Every block graph \(G\) satisfies
\[
  \chieq(G)\le L(G)+1.
\]
\end{conjecture}

Several results supported Conjecture~\ref{conj-gap-one}.  For forests, the
asserted inequality follows from a theorem of
Chang~\cite{ChangEquitableForests}.  Dybizba\'nski, Furma\'nczyk, and
Mkrtchyan~\cite{DybizbanskiFurmanczykMkrtchyan} established the same bound for
several further subclasses, including well-covered block graphs, connected
block graphs with \(\amin(G)\le2\), and block graphs in which every cut vertex
belongs to exactly two blocks.  The authors also verified the conjecture
computationally for all block graphs of order at most \(19\).  Further
structural and algorithmic results on equitable colorings of block graphs were
obtained by Furma\'nczyk and
Mkrtchyan~\cite{FurmanczykMkrtchyanAlgorithmic,FurmanczykMkrtchyanGeneral}.

We show that Conjecture~\ref{conj-gap-one} is false, and in fact that no bound of the form
\(L(G)+C\), with \(C\) an absolute constant, can hold for all block graphs.

\begin{theorem}\label{thm-main}
For every pair of integers \(d\ge2\) and \(k\ge4d-1\), there is a connected
block graph \(G_{d,k}\) such that
\[
  L(G_{d,k})=k
  \qquad\text{and}\qquad
  \chieq(G_{d,k})=k+d.
\]
\end{theorem}

Taking \(d\) arbitrarily large gives the following immediate consequence.

\begin{corollary}
The difference \(\chieq(G)-L(G)\) is unbounded on connected block graphs.
\end{corollary}

The construction exploits an obstruction not seen by \(L(G)\).  We create two
adjacent vertices \(v\) and \(w\) whose largest independent sets become too
small after deleting one pendant vertex \(u\).  Since the color classes of
\(v\) and \(w\) are distinct, and at most one of them can contain \(u\), one
of these two classes must remain small.  A large central clique then makes the
equitable lower class size too large for all \(q<k+d\), while the clique
number and the parameter \(\amin(G)\) still give only \(L(G)=k\).  The attached
large blocks make this obstruction compatible with an explicit equitable
\((k+d)\)-coloring.

\section{The construction}

Fix integers \(d\ge2\) and \(k\ge4d-1\), and let \(m=4d-1\).  Start with a
central clique \(Q\cong K_m\) containing three distinguished vertices
\(v\), \(w\), and \(c\).  Attach two \(K_k\)-blocks at \(v\), two \(K_k\)-blocks at \(w\),
and one pendant edge \(cu\).  All
vertices of the attached \(K_k\)-blocks outside \(Q\) are called
\emph{private}.  The private vertex sets are pairwise disjoint.  Let the
resulting graph be \(G_{d,k}\).  The edge \(cu\) is a \(K_2\)-block, so
\(G_{d,k}\) is a connected block graph.  It has order
\[
  |V(G_{d,k})|
  =
  m+4(k-1)+1
  =
  4k+4d-4
  =
  4(k+d-1).
\]
The choice \(m=4d-1\) makes the order exactly \(4(k+d-1)\).
Consequently, every color class in an equitable coloring with fewer than
\(k+d\) colors would have size at least four.  The assumption
\(k\ge4d-1=m\) also ensures that the attached \(K_k\)-blocks are maximum
cliques, so \(\omega(G_{d,k})=k\).

\begin{lemma}\label{lem-alpha-min}
The graph \(G_{d,k}\) has \(\amin(G_{d,k})=4\).
\end{lemma}

\begin{proof}
An independent set containing \(v\) contains no other vertex of \(Q\), no
private vertex from either \(K_k\)-block attached at \(v\), at most one private
vertex from each \(K_k\)-block attached at \(w\), and possibly the pendant
vertex \(u\).  Hence \(\alpha(G_{d,k},v)\le4\).  Equality is attained by
taking \(v\), \(u\), and one private vertex from each of the two
\(K_k\)-blocks attached at \(w\).  Thus \(\alpha(G_{d,k},v)=4\).
By symmetry, \(\alpha(G_{d,k},w)=4\).

Every other vertex lies in an independent set of size at least five.  The
vertex \(c\) can be taken together with one private vertex from each of the
four \(K_k\)-blocks.  Any vertex of \(Q\setminus\{v,w,c\}\) can be taken together with
\(u\) and one private vertex from each of the four \(K_k\)-blocks.  The vertex
\(u\) can be taken together with a vertex of \(Q\setminus\{v,w,c\}\) and one private
vertex from each of the four \(K_k\)-blocks.  Finally, if \(x\) is private in
one of the attached \(K_k\)-blocks, then \(x\), \(u\), a vertex of
\(Q\setminus\{v,w,c\}\), and one private vertex from each of the other three attached
blocks form an independent set of size six.

Therefore the minimum of the values \(\alpha(G_{d,k},x)\) is attained at
\(v\) and \(w\), and is equal to \(4\).
\end{proof}

\begin{lemma}\label{lem-lower-bound-value}
The graph \(G_{d,k}\) has \(L(G_{d,k})=k\).
\end{lemma}

\begin{proof}
By Lemma~\ref{lem-alpha-min}, \(\amin(G_{d,k})=4\), and we have already noted
that \(\omega(G_{d,k})=k\).  Moreover,
\[
  \ceil{\frac{|V(G_{d,k})|+1}{5}}
  =
  \ceil{\frac{4k+4d-3}{5}}
  \le k,
\]
because \(k\ge4d-1\).  Thus
\[
  L(G_{d,k})=
  \max\left\{
    k,
    \ceil{\frac{4k+4d-3}{5}}
  \right\}
  =
  k.\qedhere
\]
\end{proof}

The lower bound on the equitable chromatic number comes from a small
obstruction which is useful to isolate.

\begin{lemma}\label{lem-shared-augmenter}
Let \(x\) and \(y\) be adjacent vertices of a graph \(G\), let
\(z\in V(G)\setminus\{x,y\}\), and let \(a\ge1\).  Suppose
\[
  \alpha(G\setminus\{z\},x)\le a-1
  \qquad\text{and}\qquad
  \alpha(G\setminus\{z\},y)\le a-1.
\]
If \(G\) has an equitable \(q\)-coloring, then
\[
  \floor{\frac{|V(G)|}{q}}\le a-1.
\]
\end{lemma}

\begin{proof}
In any proper coloring, the color classes containing \(x\) and \(y\) are
distinct.  At least one of these two classes does not contain \(z\).  Choose
such a class.  It is an independent set in \(G\setminus\{z\}\) containing
either \(x\) or \(y\), and hence has size at most \(a-1\).  In an equitable
\(q\)-coloring every color class has size at least \(\floor{|V(G)|/q}\),
giving the claim.
\end{proof}

\begin{lemma}\label{lem-no-fewer-colors}
The graph \(G_{d,k}\) has no equitable \(q\)-coloring with \(q<k+d\).
\end{lemma}

\begin{proof}
Suppose for a contradiction that \(G_{d,k}\) has an equitable \(q\)-coloring
with \(q<k+d\).  Delete the pendant vertex \(u\).  In
\(G_{d,k}\setminus\{u\}\), an independent set containing \(v\) may use at most
one private vertex from each of the two \(K_k\)-blocks attached at \(w\), and no
other vertex of \(Q\).  This bound is attained by taking \(v\) and one private
vertex from each of those two blocks.  Hence
\(\alpha(G_{d,k}\setminus\{u\},v)=3\).  Similarly,
\(\alpha(G_{d,k}\setminus\{u\},w)=3\).
Since \(v\) and \(w\) are adjacent, Lemma~\ref{lem-shared-augmenter}, with
\(z=u\) and \(a=4\), applied to this coloring gives
\[
  \floor{\frac{|V(G_{d,k})|}{q}}\le3.
\]

On the other hand, \(q<k+d\) gives \(q\le k+d-1\), and so
\[
  \frac{|V(G_{d,k})|}{q}
  \ge
  \frac{4(k+d)-4}{k+d-1}
  =
  4.
\]
Thus \(\floor{|V(G_{d,k})|/q}\ge4\), a contradiction.  Thus, the claim
follows.
\end{proof}

It remains to give an equitable coloring with exactly \(k+d\) colors.

\begin{lemma}\label{lem-upper-coloring}
The graph \(G_{d,k}\) has an equitable \((k+d)\)-coloring.
\end{lemma}

\begin{proof}
Let \(q=k+d\), and let us write \([q]=\{1,\ldots,q\}\) for the color set.
Color the central clique \(Q\) injectively with colors \(1,\ldots,m\), so that
\(v\) receives color \(1\), \(w\) receives color \(2\), and \(c\) receives
color \(3\).  Color the pendant vertex \(u\) with color \(1\).

Let \(Y=\{3,4,\ldots,m-1\}\cup\{m+1,m+2\}\).  Since \(q=k+d\), \(k\ge m\), and
\(d\ge2\), we have \(q\ge m+2\), and hence \(Y\subseteq[q]\).  Also
\(|Y|=4d-2\).  Partition \(Y\) into disjoint sets \(Y_1\), \(Y_2\), \(Y_3\),
and \(Y_4\) with
\[
  |Y_1|=|Y_2|=d-1,
  \qquad
  |Y_3|=|Y_4|=d.
\]
For the private vertices of the two \(K_k\)-blocks attached at \(v\), use
bijectively the color sets
\[
  [q]\setminus(\{1,m\}\cup Y_1)
  \qquad\text{and}\qquad
  [q]\setminus(\{1,m\}\cup Y_2).
\]
For the private vertices of the two \(K_k\)-blocks attached at \(w\), use
bijectively the color sets
\[
  [q]\setminus(\{2\}\cup Y_3)
  \qquad\text{and}\qquad
  [q]\setminus(\{2\}\cup Y_4).
\]
Each omitted set has size \(d+1\), so each displayed private-color set has
size
\[
  q-(d+1)=k-1.
\]
The two sets used at \(v\) avoid color \(1\), and the two sets used at \(w\)
avoid color \(2\).  Thus every attached \(K_k\)-block is properly colored.
The central clique is properly colored, and \(u\), colored \(1\), is adjacent
only to \(c\), colored \(3\).  Hence the whole coloring is proper.

We now count the color-class sizes.  Color \(1\) occurs on \(v\), on \(u\),
and once in each of the two \(K_k\)-blocks attached at \(w\), so its color
class has size \(4\).  Color \(2\) occurs on \(w\) and once in each of the two
\(K_k\)-blocks attached at \(v\), so its color class has size \(3\).  Every
color in \(\{3,\ldots,m-1\}\) has one central occurrence and is omitted from
exactly one attached block, so its color class has size \(4\).  Color \(m\)
has one central occurrence and is omitted from both blocks attached at \(v\),
so its color class has size \(3\).  The colors \(m+1\) and \(m+2\) have no
central occurrence and are each omitted from exactly one attached block, so
each color class has size \(3\).  Every remaining color occurs in all four
attached \(K_k\)-blocks, and hence its color class has size \(4\).

Thus exactly four color classes, namely those of colors \(2\), \(m\), \(m+1\),
and \(m+2\), have
size \(3\), and all other color classes have size \(4\).  Since
\[
  |V(G_{d,k})|=4(k+d)-4=4q-4,
\]
this is an equitable \(q\)-coloring.
\end{proof}

\begin{proof}[Proof of Theorem~\ref{thm-main}]
Lemma~\ref{lem-lower-bound-value} gives \(L(G_{d,k})=k\).  By
Lemma~\ref{lem-no-fewer-colors}, no equitable coloring exists with fewer than
\(k+d\) colors.  By Lemma~\ref{lem-upper-coloring}, an equitable
\((k+d)\)-coloring exists.  Therefore
\(\chieq(G_{d,k})=k+d=L(G_{d,k})+d\).
\end{proof}

Since \(d\) is arbitrary, the difference
\(\chieq(G_{d,k})-L(G_{d,k})\) can be made arbitrarily large on connected
block graphs.  In particular, Conjecture~\ref{conj-gap-one} is false, and
there is no absolute constant \(C\) such that \(\chieq(G)\le L(G)+C\) for
every connected block graph \(G\).

The family remains structurally restricted.  Every \(G_{d,k}\) has exactly
six blocks, exactly three cut vertices, diameter \(3\),
\(\alpha(G_{d,k})=6\), and \(\amin(G_{d,k})=4\).  Its block-cut-tree shape is
independent of \(d\) and \(k\), and no cut vertex belongs to more than three
blocks.  To see that \(\alpha(G_{d,k})=6\), observe that an independent set
contains at most one vertex of \(Q\), at most one private vertex from each of
the four attached \(K_k\)-blocks, and possibly \(u\).  The proof of
Lemma~\ref{lem-alpha-min} exhibits an independent set of size six.

Consequently, the difference \(\chieq(G)-L(G)\) remains unbounded even under
all these restrictions.  This contrasts with the result of Dybizba\'nski,
Furma\'nczyk, and
Mkrtchyan~\cite{DybizbanskiFurmanczykMkrtchyan} that
Conjecture~\ref{conj-gap-one} holds for block graphs in which every cut vertex
belongs to exactly two blocks.  Thus allowing a cut vertex to belong to three
blocks already permits the difference to be unbounded.

The first member of the family which violates the conjectured
\(L(G)+1\) bound is obtained by taking \(d=2\) and \(k=7\).  This graph has
\(|V(G_{2,7})|=32\), \(L(G_{2,7})=7\), and \(\chieq(G_{2,7})=9\).

Although no absolute additive bound exists, it remains open whether the
equitable chromatic number can be controlled multiplicatively by \(L(G)\).
Is there a constant \(C>0\) such that every connected block graph \(G\)
satisfies \(\chieq(G)\le C L(G)\)?  The graph \(G_{2,7}\) shows that any such
constant must satisfy \(C\ge9/7\).  Moreover, setting \(k=4d-1\) gives
\(\chieq(G_{d,4d-1})/L(G_{d,4d-1})=(5d-1)/(4d-1)\to5/4\) as
\(d\to\infty\).

\medskip
\noindent\textbf{Acknowledgments.}
The author thanks the anonymous referee for helpful comments that improved the
presentation of the results.

\section*{Declaration of generative AI and AI-assisted technologies in the manuscript preparation process}

During the preparation of this work, the author used OpenAI ChatGPT in order to
assist with language editing and manuscript polishing. After using this tool,
the author reviewed and edited the content as needed and takes full
responsibility for the content of the published article.

\end{document}